\documentclass[11pt]{amsart}
\usepackage{amsmath, amsthm, amssymb}
\usepackage{graphicx}
\usepackage[usenames]{color}
\usepackage{srcltx} 
\usepackage{verbatim}
\usepackage{tikz-cd}
\usepackage{tikz}
\usetikzlibrary{positioning}
\usepackage{mathtools}
\DeclarePairedDelimiter\abs{\lvert}{\rvert}%
\DeclarePairedDelimiter\norm{\lVert}{\rVert}%

\makeatletter
\let\oldabs\abs
\def\abs{\@ifstar{\oldabs}{\oldabs*}}
\let\oldnorm\norm
\def\norm{\@ifstar{\oldnorm}{\oldnorm*}}
\makeatother

\graphicspath{{images/}}

\newtheorem{thm}{Theorem}[section]
\newcommand{\bt}{\begin{thm}}
\newcommand{\et}{\end{thm}}

\newtheorem{cor}[thm]{Corollary}   
\newcommand{\bc}{\begin{cor}}
\newcommand{\ec}{\end{cor}}

\newtheorem{lem}[thm]{Lemma}   
\newcommand{\bl}{\begin{lem}}
\newcommand{\el}{\end{lem}}

\newtheorem{prop}[thm]{Proposition}
\newcommand{\bp}{\begin{prop}}
\newcommand{\ep}{\end{prop}}

\newtheorem{defn}[thm]{Definition}
\newcommand{\bd}{\begin{defn}}    
\newcommand{\ed}{\end{defn}}

\newtheorem{rmrk}[thm]{Remark}   

\newcommand{\br}{\begin{rmrk}}
\newcommand{\er}{\end{rmrk}}

\newcommand{\be}{\begin{equation}}

\newcommand{\ee}{\end{equation}}

\newcommand{\R}{\mathbb{R}}

\DeclareMathOperator{\Vol}{Vol}
\DeclareMathOperator{\Diam}{Diam}

\newcommand{\Lip}{\operatorname{Lip}}

\newcommand{\lp}{\left (}
\newcommand{\rp}{\right )}

\def\Xint#1{\mathchoice
{\XXint\displaystyle\textstyle{#1}}%
{\XXint\textstyle\scriptstyle{#1}}%
{\XXint\scriptstyle\scriptscriptstyle{#1}}%
{\XXint\scriptscriptstyle\scriptscriptstyle{#1}}%
\!\int}
\def\XXint#1#2#3{{\setbox0=\hbox{$#1{#2#3}{\int}$ }
\vcenter{\hbox{$#2#3$ }}\kern-.6\wd0}}

\def\dashint{\Xint-}

\begin{document}

\title{Exploring a Modification of $d_p$ Convergence}

\author{Brian Allen}
\address{Lehman College, CUNY}
\email{brianallenmath@gmail.com}

\author{Edward Bryden$^\dagger$}
\address{Universiteit Antwerpen}
\email{etbryden@gmail.com}
\thanks{$^\dagger$funded by FWO grant 12F0223N}


\begin{abstract}
 In the work by M. C. Lee, A. Naber, and R. Neumayer \cite{LNN} a beautiful $\varepsilon$-regularity theorem is proved under small negative scalar curvature and entropy bounds. In that paper, the $d_p$ distance for Riemannian manifolds is introduced and the quantitative stability results are given in terms of this notion of distance, with important examples showing why other existing notions of convergence are not adequate in their setting. Due to the presence of an entropy bound, the possibility of long, thin splines forming along a sequence of Riemannian manifolds whose scalar curvature is becoming almost positive is ruled out. It is important to rule out such examples since the $d_p$ distance is not well behaved in the presences of splines that persist in the limit. Since there are many geometric stability conjectures \cite{GromovFour, SormaniIAS} where we want to allow for the presence of splines that persist in the limit, it is crucial to be able to modify the $d_p$ distance to retain its positive qualities and prevent it from being sensitive to splines. In this paper we explore one such modification of the $d_p$ distance and give a theorem which allows one to estimate the modified $d_p$ distance, which we expect to be useful in practice.
\end{abstract}

\maketitle

\section{Introduction}

In the work by M. C. Lee, A. Naber, and R. Neumayer \cite{LNN} a beautiful $\varepsilon$-regularity theorem is proved under small negative scalar curvature and entropy bounds. In that paper, the $d_p$ distance for rectifiable Riemannian manifolds is introduced and the quantitative stability results are given in terms of this notion of distance, with important examples showing why other existing notions of convergence are not adequate in their setting. In particular, many examples were given in dimension $n \ge 4$ where shortcuts develop along the sequence, that are not well behaved along Gromov-Hausdroff or Sormani-Wenger Intrinsic Flat convergence, but which are well behaved under the $d_p$ distance. More recently, D. Kazaras and K. Xu \cite{KK} showed that similarly examples, developing what they call drawstings, can exist in dimension $n=3$ with scalar curvature bounds as well. In \cite{LNN}, the presence of an entropy bound rules out the possibility of long, thin splines forming along a sequence of Riemannian manifolds whose scalar curvature is becoming almost positive. This is important since the $d_p$ distance is not well behaved in the presences of splines that persist in the limit. In particular, the distance between points along the spline will become arbitrarily large along such a sequence.

Since there are many geometric stability conjectures stated, by M. Gromov \cite{GromovFour} and C. Sormani \cite{SormaniIAS}, where we want to allow for the presence of long thin splines which persist in the limit, it is important to be able to modify the $d_p$ distance to retain its positive qualities and prevent it from being sensitive to splines. In this paper we explore one such possible modification of the $d_p$ distance and give a theorem which allows one to estimate the modified $d_p$ distance which we expect to be useful in practice.

Let $M_0=(M,g)$ be a Riemannian manifold. We can define the $d_{p,g}$ distance of M. C. Lee, A. Naber, and R. Neumayer \cite{LNN} to be
\begin{align}\label{def:dp}
d_{p,g}(x,y)=\sup \left\{|f(x)-f(y)|: \int_M |\nabla^g f|_g^p dV_{g} \le 1 \right\}.
\end{align} 

One can notice that if a spline develops along a sequence of Riemannian manifolds then the Sobolev constant for $p > n$ of such a sequence cannot be well controlled. In particular, the constant which appears in the Morrey inequality will degenerate and hence the class of functions used to define the $d_{p,g}$ metric will have unbounded differences $|f(x)-f(y)|$ for points $x,y \in M$ along the spline, despite having a controlled $L^p$ norm of the gradient. This motivates our modification of the $d_p$ distance where we control the distances $|f(x)-f(y)|$ by imposing a H\"{o}lder bound in terms of a background metric:

\begin{defn} Let $M$ be a smooth manifold, $g_0$ a smooth Riemannian manifold, and $g$ a continuous Riemannian manifold. Then we define the modified $d_p$ distance between points $x,y \in M$

\begin{align}\label{def:Modifieddp2}
d&_{p,g,g_0}^D(x,y)=
\sup \left\{|f(x)-f(y)|: \mathcal{H}_{p,g_0}(f) \le D, \int_M |\nabla^g f|_g^p dV_{g} \le 1 \right\},
\end{align} 
where
\begin{align}
    \mathcal{H}_{p,g_0}(f)=\sup_{x,y\in M}\frac{|f(x)-f(y)|}{d_{g_0}(x,y)^{\frac{p-n}{p}}}.
\end{align}
\end{defn}
Note that a reasonable choice of $D$ is given by $D =\frac{\Diam(M,g)}{\Diam(M,g_0)^{\frac{p-n}{p}}} $, but we have the freedom to choose this constant as we wish. We also immediately note that if $g_{0,\lambda}=\lambda^2 g_0$ and $g_{\lambda}=\lambda^2g$ then $d_{p,g_{\lambda},g_{0,\lambda}}^D=\lambda^{\frac{p-n}{p}}d_{p,g,g_0}^{D}$. Observe that the $\Diam(M,d_{p,g,g_0}^D) \le D$. In fact, due to the H{\"o}lder restriction on the admissible functions, we see that $d^{D}_{p,g,g_0}(x,y)\leq Dd_{g_0}(x,y)^{\frac{p-n}{p}}$. It follows by Arzela-Ascoli that the space of functions $d^{D}_{p,g,g_0}$ is a priori compact as $g$ varies, defining a  pseudometric as a limit, but may not define a distance function in the limit if tending to zero (See Lemma \ref{lem:compactness_of_distance_functions}).

By considering admissible functions whose H\"{o}lder norm with respect to $g_0$ is controlled, one can see that for an example with splines and bubbles forming along a sequence, the distance in \eqref{def:Modifieddp2} will not be sensitive to these splines or bubbles. In particular, the examples given by the first named author and C. Sormani \cite{Allen-Sormani-2} which are conformal to a sphere and develop a spline or bubble (Examples 3.5 and 3.7 of \cite{Allen-Sormani-2}, respectively) would converge to a round sphere in the modified $d_p$ sense.   In addition, the distance in \eqref{def:Modifieddp2} retains the positive features of the $d_p$ distance not being sensitive to shortcuts developing along a sequence as in the examples of \cite{KK,LNN}. We will elaborate on these conclusions after the statement of the main theorem.

A natural class of Riemannian manifolds to study using the modified $d_p$ distance is given by:

\begin{defn}\label{def:ManifoldClass}
    Given parameters $V_1,V_2,D>0$, $q_1,q_2>1$, and a background Riemannian metric $(M^n,g_0)$, $n \ge 2$, we denote by $\mathcal{N}_{g_0}(q_1,q_2, V_1,V_2,D)$ the collection of compact, connected, oriented Riemannian metric tensors $g$ on $M$ which satisfy the following:
    \begin{enumerate}
        \item $ \|g\|_{L^{\frac{q_1}{2}}(M,g_0)}\leq V_1$,
         \item $ \|g^{-1}\|_{L^{\frac{q_2}{2}}(M,g_0)}\leq V_2$,
        \item $\Diam(M,g) \le D$.
    \end{enumerate}
\end{defn}
When studying geometric stability problems related to rigidity theorems in geometric analysis, a natural background metric presents itself from the rigidity problem one is studying. For instance, in the case of the stability of Geroch's conjecture a flat torus is a natural background metric $g_0$, for the stability of the positive mass theorem Euclidean space is the natrual background metric $g_0$, and in the case of stability of Llarull's theorem the round sphere would be the background metric $g_0$. Hence for each of these stability questions there is a natural way to define the class $\mathcal{N}_{g_0}(q_1,q_2,V_1,V_2,D)$ where $q_1$ and $q_2$ are chosen depending on the problem. Our goal in this paper then is to provide a result which allows one to establish $L^p$ convergence of a sequence of inverse metrics for Riemannian manifolds in the class $\mathcal{N}(q_1,q_2, V_1,V_2,D)$ in order to conclude $d_{p,g_i,g}^D$ convergence.

\begin{thm}\label{thm:Main Theorem}
Let $M^n$ be a compact, connected, and oriented manifold, $g_i,g_0$ continuous Riemannian metrics, $p > 3n$, $n \ge 3$. Then there exists a $\bar{D}(g_0)$ so that for $D \ge \bar{D}$ if
\begin{align}
\int_M|g_j|_{g_0}^{\frac{n}{2}}dV_{g_0}&\le C,\label{eq:MainThmgBound}
\end{align}
and
\begin{align}\label{LpMetricBounds}
\int_M|g_j^{-1}-g_0^{-1}|_{g_0}^{\frac{n(p-1)}{2}}dV_{g_0} &\rightarrow 0,
\end{align}
then  
\begin{align}
 d_{p,g_j,g_0}^D \rightarrow  d_{p,g_0,g_0}^D,
\end{align}
uniformly.
\end{thm}
\begin{rmrk}
It was observed by the first named author in \cite{BALp} that if one has an $L^p$ bound $\int_M|g_j|_{g_0}^{\frac{p}{2}}dV_{g_0}\le C$ for $p > n$ then one can obtain H\"{o}lder bounds on the distance function $d_{g_j}$ in terms of $d_{g_0}$. Since examples involving splines and bubbles do not satisfy such H\"{o}lder bounds, it is important that assumption \eqref{eq:MainThmgBound} is for the critical power so that examples involving splines and bubbles are allowed in Theorem \ref{thm:Main Theorem}. In fact, the first named author and C. Sormani \cite{Allen-Sormani-2} give examples which show that under a $L^{\frac{n}{2}}$ bound on $g_j$, splines and bubbles can form. So Theorem \ref{thm:Main Theorem} shows that the modified $d_p$ distance of \eqref{def:Modifieddp2} is not sensitive to such behavior. Since M-C Lee, A. Naber, and R. Neumayer \cite{LNN} assume a entropy bound they are able to obtain $L^p$ bounds on the metric tensors for high powers of $p$ which is not surprising since an entropy bound will rule out splines and bubbles. 
\end{rmrk}

\begin{rmrk}
    In \cite{LNN}, many examples with various collapsing along a line phenomenon were given with $n \ge 4$ which $d_p$ converge to smooth Riemannian manifolds and are shown to satisfy $L^p$ bounds and convergence, similar to the requirements of Theorem \ref{thm:Main Theorem}. Recently, D. Kazaras and K. Xu \cite{KK} have shown that similar collapsing examples, called drawstrings, are possible in dimension $3$. One can check that the examples in \cite{KK} satisfy the hypotheses of Theorem \ref{thm:Main Theorem}. Hence the modified $d_p$ distance given in this paper is well suited to problems where splines, bubbles, and drawstrings occur.

    For example, in Theorem 1.19 of Lee, Naber, and Neumayer \cite{LNN} we see a torus stability result where almost non-negative scalar curvature and a bound on the entropy is assumed. The entropy bound removes the possibility of splines and bubbles forming and is generally a rather strong assumption. Using the main theorem of this paper we can trade out the entropy bound for $q_1=n$ and $q_2=n(p-1)$ for $p >3n$ which will allow for the possibility of splines and bubbles. Then one would need to figure out how to use the almost non-negativity of the scalar curvature in order to prove the inverse metric convergence of Theorem \ref{thm:Main Theorem} in order to conclude modified $d_p$ stability. Similar applications to the stability of the positive mass theorem should also be possible allowing for splines, bubbles, and drawstrings.
\end{rmrk}

\begin{rmrk}
    We note that it is also possible to choose $L^p$ bounds in the assumptions of Theorem \ref{thm:Main Theorem} which also work for $n=2$ but those powers are not conveniently stated and the main purpose of these results is to apply them to the case where $n \ge 3$. We also note that by reworking the arguments given in the proof of Theorem \ref{thm:Main Theorem} one is able to assume a power slightly under the critical power in \eqref{eq:MainThmgBound} at the cost of more complicated powers in the assumptions of Theorem \ref{thm:Main Theorem}.
\end{rmrk}

\begin{rmrk}
In the course of proving Theorem \ref{thm:Main Theorem}, in particular in Theorem \ref{LimInfEst}, we show that the hypotheses of Theorem \ref{thm:Main Theorem} also imply $L^{\frac{k}{2}}$, $k < n$ convergence of $g_j$ to $g_0$. If one is in a situation where adding $g_j \ge \left( 1-\frac{C}{j} \right) g_0$ is a natural hypothesis then by previous work of the authors \cite{AB}, we know that $d_{g_j}$ converges pointwise almost everywhere to $d_{g_0}$ with a Hausdorff measure estimate on the size of the set where pointwise convergence does not hold.
\end{rmrk}

We now give a precise definition of what it means for a sequence of Riemannian manifolds to converge in the modified $d_p$ sense. 

\begin{defn}\label{Def:Unifd_pDistance}
    Let $D>0$, $p >n$, $(M,g_0)$ be a smooth Riemannian manifold. We define the uniform modified $d_p$  distance between two continuous Riemannian metric tensors on $M$:
    \begin{align}
       d_{p,g_0}^{sup,D}((M,g_1),(M,g_2))= \sup_{x,y \in M} |d_{p,g_1,g_0}^D(x,y)- d_{p,g_2,g_0}^D(x,y)|.
    \end{align}
\end{defn}

\section{Background}

In \cite{LNN}, many useful properties of $d_p$ distance are established. Since the only modification we make here is to further restrict the class of admissible functions in the definition we will see that these important properties are retained for the modified $d_p$ distance.

\subsection{Properties of Modified $d_p$ Distance}
Let us begin by observing a simple, yet very useful, consequence of the definition of $d^{D}_{p,g,g_{0}}$.
\begin{lem}\label{lem:compactness_of_distance_functions}
		Let $(M,g_0)$ be a compact, connected, oriented Riemannian manifold, $p>n$, and $D>0$. Then, the space of distance functions $d^{D}_{p,g,g_{0}}$ is pre-compact in the space of Riemannian metrics $g$ on $M$.
	\end{lem}
	\begin{proof}
		It follows from the definition that for any $g$ we have
			\begin{equation}
				d^{D}_{p,g,g_{0}}(x,y)\leq Dd_{g_0}(x,y)^{\frac{p-n}{p}}.
			\end{equation}
		Thus, the family satisfies the hypotheses of the Arzela-Ascoli theorem.
	\end{proof}

We would like to be able to show that the distances themselves converge as $p \rightarrow \infty$. In Remark 2.33 of \cite{LNN} they mention that for reasonable spaces one should find $d_p \rightarrow d$ as $p \rightarrow \infty$. We start by showing that pointwise convergence always holds for $p \rightarrow \infty$.

	\begin{lem}\label{PointwiseConvergence}
		Let $(M,g)$, $(M,g_0)$ be compact, connected, oriented Riemannian manifolds, and $(M,d_{p,g,g_0})$ the corresponding modified $d_p$ distance metric space. If $d_g \le D_0 d_{g_0}$ then for $x_1,x_2 \in M$ and every $D \ge D_0+2$ we find
			\begin{align}
    			d_{p,g,g_0}^{D}(x_1,x_2) \rightarrow d_g(x_1,x_2),
			\end{align}
		as $p \rightarrow \infty$.
	\end{lem}
	\begin{proof}
		For each $q > n$ consider a test function $f_q$ so that
			\begin{align}
    			d_{q,g,g_0}^{D}(x_1,x_2) = |f_q(x_1)-f_q(x_2)|+\varepsilon.
			\end{align}
		By definition, we have that
		\begin{equation}
			\int_{M}\abs{\nabla f_q}^qdV_{g}\leq 1
		\end{equation}
		and
		\begin{equation}
			\mathcal{H}_{q,g_0}(f_q)\leq D.
		\end{equation}

		Now, by H{\"o}lder's inequality for $1< p<q<\infty$
			\begin{equation}
    			\left(\int_{M}\abs{\nabla^g f_q}^pdV_{g}\right)^{\frac1p}\leq\abs{M}^{\frac1p-\frac1q}\left(\int_{M}\abs{\nabla^g f_q}^{q}dV_{g}\right)^{\frac1q}
			\end{equation}
		which implies that for $q > p$ we have
			\begin{equation}
    			\|\nabla^g f_{q}\|_{L^p}\leq\abs{M}^{\frac1p-\frac1q}
			\end{equation}
		Therefore, we see that for any fixed $p$ we have
			\begin{equation}
    			\limsup_{q\rightarrow\infty}\|\nabla^g f_q\|_{L^p}\leq\abs{M}^{\frac1p}
			\end{equation}

		Next, by the Poincare inequality we know
			\begin{align}\label{LpBound}
   				\dashint_M|f_q - \bar{f}_q|^pdV_g \le C_p\dashint_M |\nabla^g f_q|_g^pdV_g \le C_p,
			\end{align}
		and hence $\|f_q-\bar{f}_q\|_{W^{1,p}}\le C$. So we can extract a subsequence so that
			\begin{align}
  				f_q-\bar{f}_q \rightarrow f_{\infty}, \quad   \nabla^g f_q \rightharpoonup h_{\infty}, \quad \text{ in } L^p,\quad\text{for all $1<p<q$.}
			\end{align}
		By successively picking subsequences as $p \rightarrow \infty$ and using the lower semi-continuity of the norm under weak convergence, it follows that for all $p$ we have
			\begin{equation}
    			\|h_{\infty}\|_{L^p}\leq\abs{M}^{\frac1p}
			\end{equation}
		Thus, using the fact that $\displaystyle \lim_{p\rightarrow\infty}\|h_{\infty}\|_{L^p}=\|h_{\infty}\|_{L^{\infty}}$, we see that $\|h_{\infty}\|_{L^{\infty}}\leq1$. In particular, since the weak gradient of $f_{\infty}$ is $h_{\infty}$, we see that $f_\infty$ is 1-Lipschitz.
		Next, as each $f_{q}-\bar{f}_{q}$ converges to $f_{\infty}$ in $W^{1,p}$ for all $p$, by picking $p$ to be large enough, we can ensure that this convergence is uniform by applying Morrey's inequality to $f_{\infty}-\bigl(f_{q}-\bar{f}_{q}\bigr)$. In particular, the convergence is pointwise, so
			\begin{equation}
    			\lim_{q\rightarrow\infty}|f_{q}(x)-f_{q}(y)|=\lim_{q\rightarrow\infty}|f_{q}(x)-\bar{f}_q-(f_{q}(y)-\bar{f}_q)|=|f_{\infty}(x)-f_{\infty}(y)|.
			\end{equation}
		Thus, for any $\varepsilon>0$ we have that
			\begin{equation}
				\limsup_{q\rightarrow\infty}d_{q,g,g_0}^{ D}(x,y)\leq d_{g}(x,y)+\varepsilon. 
			\end{equation}
		On the otherhand, it is well known that
			\begin{equation}
				d_{g}(x,y)=\sup\left\{\abs{f(x)-f(y)}:\Lip_{d_{g}}(f)\leq1\right\}.
			\end{equation}
		Let $f_{\varepsilon}$ be a 1-Lipschitz function with respect to $d_g$ such that
		\begin{equation}
			d_{g}(x,y)\leq\abs{f_{\varepsilon}(x)-f_{\varepsilon}(y)}+\varepsilon.
		\end{equation}
		Then, for all $p$ we have that $\tilde{f}_{p,\varepsilon}=\abs{M}^{-\frac{1}{p}}f_{\varepsilon}$ satisfies
		\begin{equation}
			\int\abs{\nabla\tilde{f}_{p,\varepsilon}}^pdV_{g}\leq 1
		\end{equation}
		and
		\begin{equation}
			\mathcal{H}_{p,g}(\tilde{f}_{p,\varepsilon})\leq\abs{M}^{-\frac{1}{p}}\Diam_g(M)^{\frac{n}{p}}=\left(\frac{\Diam_{g}(M)^n}{\abs{M}_{g}}\right)^{\frac1p}.
		\end{equation}
        For large enough $p$, the term on the right hand side will be arbitrarily close to $1$. Since $d_g\leq D_0 d_{g_0}$ we have  for $0<\delta<1$ that $H_{p,g_0}\bigl(\tilde{f}_{p,\varepsilon}\bigr)\leq D_{0}^{\frac{p-n}{p}}+\delta\leq D$. Thus, our choice of $D$ ensures that $\Tilde{f}_{p,\varepsilon}$ is a test function for every $p$ large enough. Thus, we have that
		\begin{equation}
			d^{D}_{p,g,g_0}(x,y)\geq \bigl|\tilde{f}_{p,\varepsilon}(x)-\tilde{f}_{p,\varepsilon}(y)\bigr|\geq \abs{M}^{-\frac{1}{p}}(d_{g}(x,y)+\varepsilon).
		\end{equation}
		This gives us that
		\begin{equation}
			\liminf_{p\rightarrow\infty}d^{D}_{p,g,g_0}(x,y)\geq d_{g}(x,y)+\varepsilon,
		\end{equation}
		and so the result follows.
	\end{proof}
\begin{cor}
    Let $(M,g)$, $(M,g_0)$ be compact, connected, oriented Riemannian manifolds, and $(M,d_{p,g,g_0})$ the corresponding modified $d_p$ distance metric space. If $d_g \le D_0 d_{g_0}$ and $D \ge D_0+2$, then $d^D_{p,g,g_0}$ converges uniformly to $d_g$ as $p$ tends to infinity.
\end{cor}
\begin{proof}
 Notice that the conclusion of Lemma \ref{lem:compactness_of_distance_functions} holds even if we let the exponent $p$ approach infinity. This combined with the pointwise convergence in Lemma \ref{PointwiseConvergence} shows that the convergence is actually uniform.
\end{proof}

We now make a simple observation about how the modified $d_{p,g,g_0}^D$ distance scales. 

\begin{lem}
If $g_{\lambda}=\lambda^2g$ and $g_{0,\lambda}=\lambda^2g_0$ then we find that for all $x_1,x_2 \in M$ 
\begin{align}
d_{p,g_{\lambda},g_{0,\lambda}}(x_1,x_2) = \lambda^{\frac{p-n}{p}} d_{p,g,g_0}(x_1,x_2).
\end{align}
\end{lem}
\begin{proof}
We first observe that
\begin{align}
 \int_M |\nabla^g f|_{g_{\lambda}}^p dV_{g_{\lambda}}= \lambda^{n-p}\int_M |\nabla^g f|_g^p dV_{g}.
\end{align}
Hence if $f$ is the admissible function so that
\begin{align}
d_{p,g,g_0}(x_1,x_2)=|f(x_1)-f(x_2)|-\varepsilon,
\end{align}
we see that for $f_{\lambda}=\lambda^{\frac{p-n}{p}}f$ is admissible for $g_{\lambda}$, since additionally $\mathcal{H}_{p,g_{0,\lambda}}(f_{\lambda})\le D$, and so 
\begin{align}
d_{p,g_{\lambda},g_{0,\lambda}}(x_1,x_2)&\ge|f_{\lambda}(x_1)-f_{\lambda}(x_2)|
\\&=\lambda^{\frac{p-n}{p}}|f(x_1)-f(x_2)|
\\&=\lambda^{\frac{p-n}{p}}(d_{p,g,g_0}(x_1,x_2)+\varepsilon).
\end{align}
Since this is true for all $\varepsilon>0$ we find 
\begin{align}
d_{p,g_{\lambda},g_{0,\lambda}}(x_1,x_2) \ge \lambda^{\frac{p-n}{p}} d_{p,g,g_0}(x_1,x_2).
\end{align}
By reworking the argument from $g_{\lambda}$ to $g$ a second time one obtains the desired result.
\end{proof}

We now note a convergence result where we assume smooth convergence. This will be used in the following proposition.

\begin{lem}\label{lem:Smoothd_p}
    Let $(M,g_{i})$ and $(M,h_{i})$ be two sequences of compact, connected, oriented Riemannian manifolds such that $g_{i}$ smoothly converges to $g_{\infty}$ and $h_{i}$ smoothly converges to $h_{\infty}$. Then, we have that $d^{D}_{p,g_i,h_i}$ uniformly converges to $d^{D}_{p,g_{\infty},h_{\infty}}$. 
\end{lem}
\begin{proof}
    First, notice that the conclusion of Lemma \ref{lem:compactness_of_distance_functions} holds in this scenario, so that the family of metric functions is pre-compact. Therefore, in order to establish the result, it suffice to establish pointwise convergence.

    We begin by showing that for any $x,y$ in $M$ we have
    \begin{equation}\label{eq:easy_liminf}
        \liminf_{i\rightarrow\infty}d^{D}_{p,g_i,h_i}(x,y)\geq d^{D}_{p,g_{\infty},h_{\infty}}(x,y).
    \end{equation}
    To see this, consider any test function, say $f$, for the distance $d^{D}_{p,g_{\infty},h_{\infty}}$. Since both $g_i$ and $h_i$ converge smoothly to $g_{\infty}$ and $h_{\infty}$, respectively, it follows that one may pick a sequence of numbers $C_{i}$ tending towards $1$ such that $f_{i}=C_{i}f$ are test functions for $d^{D}_{p,g_i,h_i}$. Since $d^{D}_{p,g_i,h_i}(x,y)\geq \abs{f_i(x)-f_{i}(y)}$, we get \eqref{eq:easy_liminf}.

    Observe that since $g_i$ smoothly converges $g_{\infty}$, so do the inverse metrics. In particular, we can choose a subsequence such that
    \begin{equation}
        (1-i^{-1})^2\leq\frac{g^{-1}_{i}(df,df)}{g^{-1}_{\infty}(df,df)}\leq(1+i^{-1})^2
    \end{equation}
    for all $f$. Thus, we see that
    \begin{equation}
        (1-i^{-1})\leq\frac{\abs{\nabla^{g_i}f}_{g_i}}{\abs{\nabla^{g_{\infty}}f}_{g_{\infty}}}\leq(1+i^{-1})
    \end{equation}
    for all $f$.
    Similarly, we see that
    \begin{equation}
        (1+i^{-1})^{-n}\leq\frac{dV_{g_i}}{dV_{g_{\infty}}}\leq(1-i^{-1})^{-n}
    \end{equation}
    Now, let $f_{i}$ be a sequence of test functions for $d^{D}_{p,g_i,h_i}$ such that
    \begin{equation}
        d^{D}_{p,g_i,h_i}(x,y)\leq\abs{f_i(x)-f_{i}(y)}+i^{-1}.
    \end{equation}
    Then the sequence $f_{i}$ is uniformly bounded in $W^{1,p}(g_{\infty})$, and so a subsequence weakly converges to some limit $f_{\infty}$. Additionally, since the $h_i$ smoothly converge to $h_{\infty}$, we see that $\mathcal{H}_{p,h_{\infty}}(f_i)$ is uniformly bounded and hence an additional subsequence converges uniformly to $f_{\infty}$.  Finally, the lower-semicontinuity of the norm now implies that $f_{\infty}$ is a test function. Hence we see that
    \begin{equation}
       \limsup_{i\rightarrow \infty} d^{D}_{p,g_i,h_i}(x,y)\leq\abs{f_{\infty}(x)-f_{\infty}(y)}\leq d^{D}_{p,g_{\infty},h_{\infty}}(x,y).
    \end{equation}
\end{proof}

Now we would like to know whether or not having isometric modified $d_p$ distances is enough to know that the Riemannian metrics are isometric.
\begin{prop}
    Let $(M,g)$, $(M,h)$, and $(M,g_0)$ be compact , connected, oriented Riemannian manifolds. Then for $p > n$, $D>0$ if $(M,d_{p,g,g_0}^D)$, $(M,d_{p,h,g_0}^D)$ are isometric metric spaces then $(M,g)$ and $(M,h)$ are isometric Riemannian manifolds.
\end{prop}
\begin{proof}
    The same proof as in Proposition 2.34 of \cite{LNN} works in our case with the following modifications. At the step where they analyze the $d_{p,g_i}$ distance of rescaled metrics $g_i=d_i^{-2}g$, we analyze $d_{p,d_i^{-2}g,d_i^{-2}g_0}$. One sees that $d_{p,d_i^{-2}g,d_i^{-2}g_0}(x,y)$ is equal to $d^{\frac{n-p}{p}}_id_{p,g,g_0}(x,y)$. Next, at the step when they apply Theorem 1.7 of \cite{LNN} we would apply Lemma \ref{lem:Smoothd_p} on re-scaled metric balls using the fact that both $d^{-2}_{i}g$ and $d^{-2}_ig_0$ uniformly converge to the Euclidean metric.
\end{proof}

\subsection{Useful Technical Lemmas}

Using the theory in \cite{gigli2020lectures} one finds a useful characterization of the norm of the gradient of a function  $f\in W^{1,p}$ as the essential infimum of upper gradients of $f$. We use this characterization as an important technical point in the proof of Theorem \ref{thm:pointwiseBelow}. In fact, in the theory developed in \cite{gigli2020lectures}, it is shown that a function's minimal upper gradient has this characterization. The result then follows by showing that for Sobolev functions, the minimal upper gradient and the norm of the weak gradient coincide almost everywhere. The proof of this result for Euclidean space is the content of Corollary 2.1.24 together with Proposition 2.1.36 and Theorem 2.1.37 in \cite{gigli2020lectures}. Although \cite{gigli2020lectures} only develops the theory of $W^{1,2}$ explicitly, the machinery developed there allows one to define $W^{1,p}$ for all $p\geq1$. The proof below is a modification of the approach found in \cite{gigli2020lectures}.
\begin{prop}\label{prop: Technical Gradient Check}
    The Sobolev space defined in \cite{gigli2020lectures} and the standard Sobolev space agree on smooth Riemannian manifolds. In addition, this implies that the norm of the gradient of a function $f\in W^{1,p}$ is the essential infimum of upper gradients of $f$.
\end{prop}
\begin{proof}
    Let us first fix some notation. For $f\in W^{1,p}\cap C^{\infty}(M)$ we let $\nabla f$ denote the classical gradient of $f$. For $f$ in the Sobolev space $S^p\cap L^p$ as defined in \cite{gigli2020lectures}, we let $\abs{Df}$ denote the minimal weak upper-gradient of $f$.

    We first claim that for $f\in C^{\infty}(M)$, we have that $\abs{\nabla f}=\abs{Df}$ almost everywhere. To see this, first observe that by definition  $\abs{\nabla f}$ is an upper gradient for $f$, and so we must have that $\abs{Df}\leq \abs{\nabla f}$ almost everywhere. Let $X=\left\{x:\abs{Df}<\abs{\nabla f}\right\}$. Thus, we see that for almost every $x$ such that $\abs{\nabla f}(x)=0$, we actually have equality. So, we may focus on the set $Y=X\cap\left\{x:\abs{\nabla f}(x)>0\right\}$. For any $x$ such that $\abs{\nabla f}(x)>0$, we may find a $\delta>0$ and an open neighborhood $U$ about $x$, with compact closure, such that for all $y\in U$ we have $\abs{\nabla f}(y)>\delta$. Let $\Phi$ be the flow generated by $\frac{\nabla f}{\abs{\nabla f}}$ with initial points in $U$. Since $U$ has compact closure and $\abs{\nabla f}$ is continuous, there exists $\varepsilon>0$ such that $\Phi:U\times[0,\varepsilon]\rightarrow M$ exists, and is smooth. In fact, by picking $\varepsilon>0$ small enough we ensure that the Jacobian is non-vanishing, and so we may guarantee that the map $y\rightarrow \Phi(y,t)$ defines a test family of curves with measure $\frac{1}{\abs{U}_g}\Vol_{g}|_{U}$ in the sense of \cite{gigli2020lectures}. Consider now that on one hand
    \begin{equation}\label{eq:grad_f}
        \int_{U}f\left(\Phi(y,\varepsilon)\right)-f\left(\Phi(y,0)\right) dV_g=\int_{U}\int^{\varepsilon}_{0}\abs{\nabla f}\left(\Phi(y,t)\right)dt dV_g,
    \end{equation}
    while on the other hand we have by definition of $\abs{Df}$ that
    \begin{equation}\label{eq:upper_grad_f}
        \int_{U}f\left(\Phi(y,\varepsilon)\right)-f\left(\Phi(y,0)\right) dV_g\leq\int_{U}\int^{\varepsilon}_{0}\abs{Df}\left(\Phi(y,t)\right)dt dV_g.
    \end{equation}
    Notice that $\Phi(y,t)$ is the evaluation map on the family of curves as in \cite{gigli2020lectures}. Then, by subtracting Equation \eqref{eq:upper_grad_f} from Equation \eqref{eq:grad_f} and using the fact that $\abs{\nabla f}\geq \abs{Df}$ almost everywhere, we see $Y\cap U$ has measure zero. This in turn shows that $X$ has measure zero. Thus, $\abs{\nabla f}=\abs{Df}$ almost everywhere.

    Now, it follows directly from a partition of unity argument and the proof of Proposition 2.1.36 C) in \cite{gigli2020lectures} that for any $f\in W^{1,p}$ we may find a sequence of smooth functions $f_j$ and $G_j$ such that $f_j\rightarrow f$ in $W^{1,p}$, $G_j\rightarrow \abs{Df}$ in $L^p$ and
    \begin{equation}
        \abs{Df_j}\leq G_{j}.
    \end{equation}
    Since $\abs{Df_j}=\abs{\nabla f_j}$ almost everywhere, we actually have that
    \begin{equation}\label{eq:grad_upper_bound_f}
        \abs{\nabla f_j}\leq G_j
    \end{equation}
    almost everywhere. Since $f_{j}\rightarrow f$ and $\abs{\nabla f_{j}}\rightarrow\abs{\nabla f}$ in $L^{p}$, it follows from \cite{gigli2020lectures} that $\abs{\nabla f}$ is an upper gradient for $f$, and so $\abs{Df}\leq\abs{\nabla f}$. However, from Equation \eqref{eq:grad_upper_bound_f} the opposite inequality must also hold. Therefore, we have equality almost everywhere.
\end{proof}

We will often want to compute the norm of $g$ and $g^{-1}$ with respect to a background metric $g_0$ so we review the definition of these norms.

\begin{defn}\label{def-NormDeterminant}
    Let $g,g_0$ be Riemannian metrics so that $e_1,...,e_n$ is an orthonormal frame with respect to $g_0$ that diagonalizes $g$, i.e.
    \begin{align}
        g(e_i,e_i)=\lambda_i^2,
    \end{align}
    for eigenvalues $0\le \lambda_1 \le \lambda_2 \le ... \le \lambda_n$ then
    \begin{align}
        |g|_{g_0}&= \sqrt{\lambda_1^4+...+\lambda_n^4 },
        \\ |g^{-1}|_{g_0}&= \sqrt{ \lambda_1^{-4}+...+\lambda_n^{-4} },
    \end{align}
    and
    \begin{align}
        \det(g)_{g_0}&=\lambda_1^2...\lambda_n^2,
        \\\det(g^{-1})_{g_0}&=\lambda_1^{-2}...\lambda_n^{-2}.
    \end{align}
\end{defn}

We now establish a simple way of estimating the determinant of $g$ by the norm of $g$, with respect to $g_0$.

\begin{lem}
    Let $g,g_0$ be Riemannian metrics  then
    \begin{align}
        \det(g)_{g_0} &\le n^{-\frac{n}{2}}|g|_{g_0}^{n},
        \\ \det(g^{-1})_{g_0} &\le n^{-\frac{n}{2}}|g^{-1}|_{g_0}^{n},
    \end{align}
\end{lem}
\begin{proof}
    We will prove the first inequality and the second inequality will follow by a similar argument. By Definition \ref{def-NormDeterminant} we can consider the diagonal matrix with entries $\lambda_1^2,...,\lambda_n^2$ and by the determinant trace inequality we find
    \begin{align}
    \det(g)_{g_0}^{\frac{1}{n}}= (\lambda_1^2...\lambda_n^2)^{\frac{1}{n}} \le \frac{1}{n} (\lambda_1^2+...+\lambda_n^2),  
    \end{align}
and now we use the fact that the $\ell_1$ norm is bounded by $\sqrt{n}$ times the $\ell_2$ norm to find
\begin{align}
    \det(g)_{g_0}^{\frac{1}{n}} \le \frac{1}{n} \sqrt{n}\sqrt{\lambda_1^4+...+\lambda_n^4}=n^{-\frac{1}{2}}|g|_{g_0},  
    \end{align}
    and the result follows by raising both sides to the $n$th power.
\end{proof}

We now establish a useful observation for when one is comparing norms of gradients with respect to two different metrics.

\begin{lem}\label{lem:MetricCauchy}
Let $g$ be a Riemannian metric, $f: M \rightarrow \R$ a function, and $\omega$ a $(0,2)$ tensor then
\begin{align}
\omega(df,df) \le |\omega|_{g}|\nabla^{g}f|_{g}^2.
\end{align}
In addition, if $h$ is another Riemannian metric then
\begin{align}
    |g^{-1}|_h=|h|_g.
\end{align}
\end{lem}
\begin{proof}
We define a $(0,2)$ tensor $\nabla^{g}f \otimes \nabla^{g}f$ so that if $\alpha,\beta$ are $1-$ forms then
\begin{align}
(\nabla^{g}f \otimes \nabla^{g}f)(\alpha,\beta) = \alpha(\nabla^{g}f) \beta(\nabla^{g}f),
\end{align}
and hence for $dx_i, dx_j$ we find
\begin{align}
(\nabla^{g}f \otimes \nabla^{g}f)(dx_i,dx_j) = dx_i(\nabla^{g}f) dx_j(\nabla^{g}f)= g^{il}\partial_lf g^{jk}\partial_kf.
\end{align}
Now we notice that 
\begin{align}
g(\omega, \nabla^{g}f \otimes \nabla^{g}f) \le |\omega|_{g} |\nabla^{g}f \otimes \nabla^{g}f|_{g} = |\omega|_{g}|\nabla^{g}f|_{g}^2,
\end{align}
and lastly we see that
\begin{align}
g(\omega, \nabla^{g}f \otimes \nabla^{g}f)&=g_{pj}g_{qm}\omega^{pq} g^{ij} \partial_if g^{lm}\partial_lf
\\&=\delta_p^i\delta_q^l \omega^{pq}\partial_if\partial_lf=\omega(df,df).
\end{align}
For the second result, we can calculate
\begin{align}
    |g^{-1}|_h=h_{ij}h_{pq}g^{ip}g^{jq}=|h|_{g}.
\end{align}
\end{proof}

We now also review a simple inequality which will be used repeatedly in the next section. We suspect this is a well known theorem but give a short proof for completeness.

\begin{prop}\label{prop:reverse_triangle_ineq_snowflake}
    Let $a,b$ be in $\R$ and let $t\in(0,1]$. Then, we have that
    \begin{equation}
        \abs{\abs{a}^{t}-\abs{b}^t}\leq\abs{a-b}^t.
    \end{equation}
\end{prop}
\begin{proof}
    Let $d(x,y)=\abs{x-y}$. Then, it is well known that the snowflake metric $d_{t}(x,y)=d(x,y)^t$ satisfies all of the properties of a metric. In particular, we have the reverse triangle inequality:
    \begin{equation}
        \abs{d_t(x,a)-d_{t}(x,b)}\leq d_{t}(a,b).
    \end{equation}
    Choosing $x=0$ gives the desired result.
\end{proof}

\section{Proof of Main Theorems}
From Lemma \ref{lem:compactness_of_distance_functions}, we know that for any sequence of Riemannian metrics $g_i$, a subsequence of $d^D_{p,g_i,g_0}$ converges uniformly to a non-negative, symmetric function which satisfies the triangle inequality. We would like to show that this limiting function has to be the $d_{p,g_0,g_0}$ distance, which we will accomplish by showing pointwise convergence. Pointwise convergence will be shown in two parts where we start by showing a bound from below on pointwise convergence of the sequence.

\begin{thm}\label{LimInfEst}
Let $M^n$ be a compact, connected, oriented manifold, $g$ a smooth Riemannian metric, and $g_i$ a sequence of continuous Riemannian metrics. If $p > n$, $D>0$ and
\begin{align}
\int_M|g_j^{-1}-g_0^{-1}|_{g_0}^{\frac{n(p-1)}{2}}dV_{g_0} &\rightarrow 0\label{eq:inverse_metric_convergence},
\\ \int_M|g_j|_{g_0}^{\frac{n}{2}}dV_{g_0} & \le C\label{eq:int_metric_bound},
\end{align}
then for $x_1,x_2 \in M$ we find
\begin{align}
\liminf_{i \rightarrow \infty} d_{p,g_j,g_0}^D(x_1,x_2) \ge  d_{p,g_0,g_0}^D(x_1,x_2).
\end{align}
\end{thm}
\begin{rmrk}
    One should note that by reworking the argument one can replace the $\frac{n}{2}$ in \eqref{eq:int_metric_bound} with a $p$ close to $\frac{n}{2}$ at the expense of a more complicated power in \eqref{eq:inverse_metric_convergence}.
\end{rmrk}
\begin{proof}
If $\phi \in C^{\infty}(M,g_0)$ at each point $x \in M$ we can choose  $\left\{e_{i}^{j}\right\}_{i=1}^{n}$ an orthonormal basis for $g_0$ which diagonalizes $g_{j}$, i.e. $g_j(e_{i}^{j},e_{i}^{j})=(\lambda_{i}^{j})^2$ and $g_j(e_{i}^{j},e^j_k)=0$ for $i\not = k$, and $P_j=\sqrt{\det((g_j)_{lm})}$. Then, we have that
\begin{equation}
    \abs{\nabla^{g_{j}}\phi}_{g_j}=\sqrt{\sum_{i=1}^{n}\left(\lambda^{i}_{j}\right)^{-2}d\phi(e_{i}^{j})^2}=\abs{\sum_{i=1}^{n}\left(\lambda^{i}_{j}\right)^{-1}d\phi(e_{i}^{j})e_{i}^{j}}_{g_0}
\end{equation}
Our goal is to show that an admissible function for $g_0$ will be almost admissible for the sequence $g_j$. To this end, we see by the reverse triangle inequality
\begin{align}
&\abs{\lp\int_M\abs{\nabla^{g_{j}}\phi}_{g_j}^{p}dV_{g_i}\rp^{\frac1p}-\lp\int_M\abs{\nabla^{g_0}\phi}_{g_0}^pdV_{g_0}\rp^{\frac{1}{p}}}
    \\&=\abs{\lp\int_M\abs{P^{\frac{1}{p}}_{j}\nabla^{g_{i}}\phi}_{g_j}^{p}dV_{g_0}\rp^{\frac1p}-\lp\int_M\abs{\nabla^{g_0}\phi}_{g_0}^pdV_{g_0}\rp^{\frac{1}{p}}}
    \\&\leq\lp\int_M\abs{\sum_{i=1}^{n}\lp\lp\lambda_{i}^{j}\rp^{-1}P^{\frac1p}_{j}-1\rp d\phi(e_{i}^{j})e_{i}^{j}}^{p}_{g_{0}}dV_{g_0}\rp^{\frac{1}{p}}.\label{eq:Final1}
\end{align}
So, it follows from the Minkowski inequality  
\begin{align}
  &\lp\int_M\abs{\sum_{i=1}^{n}\lp\lp\lambda_{i}^{j}\rp^{-1}P^{\frac1p}_{j}-1\rp d\phi(e_{i}^{j})e_j^i}^{p}_{g_{0}}dV_{g_0}\rp^{\frac{1}{p}} 
  \\&\le \sum_{i=1}^{n}\lp\int_M\abs{\lp\lp\lambda_{i}^{j}\rp^{-1}P^{\frac1p}_{j}-1\rp }^{p}_{g_{0}}|d\phi(e_{i}^{j})|^pdV_{g_0}\rp^{\frac{1}{p}}
  \\&\le\|d\phi(e_{i}^{j}) \|_{L^{\infty}(g_0)}\sum_{i=1}^{n}\lp\int_M\abs{\lp\lambda_{i}^{j}\rp^{-1}P^{\frac1p}_{j}-1 }^{p}_{g_{0}}dV_{g_0}\rp^{\frac{1}{p}}.\label{eq:Final2}
\end{align}
and hence if we show
\begin{equation}
    P^{\frac{1}{p}}_{j}\lp\lambda_{i}^{j}\rp^{-1}\stackrel{L^{p}}{\rightarrow}1
\end{equation}
for each $j$, then the above converges to zero. We first notice that
\begin{equation}\label{eq:Final3}
    P^{\frac{1}{p}}_{j}\lp\lambda_{i}^{j}\rp^{-1}=\left( \prod_{m\neq i}\lambda_{m}^{j}\right)^{\frac{1}{p}}\lp\lambda_{i}^{j}\rp^{\frac{1}{p}-1}
\end{equation}
and now we want to argue that if $g_j^{-1}$ converges in some $L^p$ and $g_j$ is bounded in $L^q$ then $g_j$ converges in $L^k$ as well where we will obtain precise choices for $p,q,k$. 

To this end, we have
\begin{equation}
    \left|1-\lambda_{i}^{j}\right|^{k}=\left(\lambda_{i}^{j}\right)^{k}\left|\left(\lambda_{i}^{j}\right)^{-1}-1\right|^{k}.
\end{equation}
 As long as $k<n$ we may apply H{\"o}lder's inequality with $\frac{n}{k}$ and $\frac{n}{n-k}$ to find
\begin{equation}
    \int_M\abs{1-\lambda_{i}^{j}}^kdV_{g_0}\leq\sup_{ij}\|\lambda_{i}^{j}\|^{k}_{L^{n}}\lp\int_M\abs{(\lambda_{i}^{j})^{-1}-1}^{\frac{kn}{n-k}}dV_{g_0}\rp^{\frac{n-k}{n}}
\end{equation}
Therefore, if $\sup_{ij}\|\lambda_{i}^{j}\|_{L^{n}}<\infty$ and $(\lambda_{i}^{j})^{-1}$ converges to $1$ in $L^{\frac{kn}{n-k}}$, then the above shows that $\lambda_{i}^{j}$ converges to $1$ in $L^k$. 

Now we observe that we can rewrite
\begin{align}
    \lambda_2^j...\lambda_n^j-1&=\lambda_2^j...\lambda_n^j-\lambda_3^j...\lambda_n^j+\lambda_3^j...\lambda_n^j-...-\lambda_n^j+\lambda_n^j-1
    \\&=(\lambda_2^j-1)\lambda_3^j...\lambda_n^j+...+(\lambda_{n-1}^j-1)\lambda_n^j+(\lambda_n^j-1)
\end{align}

If we pick $n>k\ge n-1$, then by the Minkowski inequality and then the generalized H\"{o}lder inequality we find 
\begin{align}
    \| \lambda_2^j...\lambda_n^j-1\|_{L^{\frac{k}{n-1}}}& \le \|(\lambda_2^j-1)\lambda_3^j...\lambda_n^j\|_{L^{\frac{k}{n-1}}}+...+\|\lambda_n^j-1\|_{L^{\frac{k}{n-1}}}
    \\&\le \|(\lambda_2^j-1)\|_{L^{k}}\|\lambda_3^i\|_{L^{k}}...\|\lambda_n^j\|_{L^{k}}
    \\&\quad +\|(\lambda_3^j-1)\|_{L^{\frac{k(n-2)}{n-1}}}\|\lambda_4^j\|_{L^{{\frac{k(n-2)}{n-1}}}}...\|\lambda_n^j\|_{L^{{\frac{k(n-2)}{n-1}}}}
    \\&\quad +...+\|\lambda_n^j-1\|_{L^{\frac{k}{n-1}}}.
\end{align}

The same argument above applies if we replace $\lambda_1^j$ with $\lambda_{i}^{j}$ for $1 \le i \le n$ and hence we see that 
\begin{equation}\label{IntermediateConvergenceEq}
    \prod_{m\neq j}\lambda^{i}_{m}\stackrel{L^{\frac{k}{n-1}}}{\rightarrow}1.
\end{equation}
Now, for each $i$ we have
\begin{equation}
    \left\|P^{\frac{1}{p}}_{j}\left(\lambda_{i}^{j}\right)^{-1}-1\right\|_{L^{p}}=\left\|\left(\lambda_{i}^{j}\right)^{\frac{1-p}{p}}\left(\prod_{m\neq i}\lambda^j_{m}\right)^{\frac1p}-1\right\|_{L^{p}}.
\end{equation}
Adding and subtracting by $\left(\lambda_{i}^{j}\right)^{\frac{1-p}{p}}$, and then using the triangle inequality shows us that
\begin{align}
    \left\|P^{\frac{1}{p}}_{j}\left(\lambda_{i}^{j}\right)^{-1}-1\right\|_{L^{p}}\leq&\left\|\left(\lambda_{i}^{j}\right)^{\frac{1-p}{p}}\left(\left(\prod_{m\neq i}\lambda^j_{m}\right)^{\frac{1}{p}}-1\right)\right\|_{L^{p}}
    \\&+\left\|\left(\lambda_{i}^{j}\right)^{\frac{1-p}{p}}-1\right\|_{L^{p}}.
\end{align}
We now analyze the first term on the right:
\begin{align}
    &\left\|\left(\lambda_{i}^{j}\right)^{\frac{1-p}{p}}\left(\left(\prod_{m\neq i}\lambda^j_{m}\right)^{\frac1p}-1\right)\right\|_{L^{p}}
    \\\qquad&=\left(\int_M\abs{\left(\lambda_{i}^{j}\right)^{-1}}^{p-1}\abs{\left(\prod_{m\neq i}\lambda^j_{m}\right)^{\frac1p}-1}^{p}dV_{g_0}\right)^{\frac1p}.\label{eq:RHS}
\end{align}
By applying Proposition \ref{prop:reverse_triangle_ineq_snowflake} to \eqref{eq:RHS} gives the following bound on \eqref{eq:RHS}
\begin{equation}
    \left(\int_M\abs{\left(\lambda_{i}^{j}\right)^{-1}}^{p-1}\abs{\prod_{m\neq i}\lambda^j_{m}-1}dV_{g_0}\right)^{\frac1p}.\label{eq:PreviousTerm}
\end{equation}
If $n-1<k<n$, then we may use H{\"o}lder's inequality with exponents $\frac{k}{n-1}$ and $\frac{k}{k+1-n}$ to show that \eqref{eq:PreviousTerm} is bounded above by
\begin{equation}
    \left(\int_M\abs{\left(\lambda_{i}^{j}\right)^{-1}}^{\frac{k(p-1)}{k+1-n}}dV_{g_0}\right)^{\frac{k+1-n}{pk}}\left(\int_M\abs{\prod_{m\neq i}\lambda^j_{m}-1}^{\frac{k}{n-1}}dV_{g_0}\right)^{\frac{n-1}{pk}}\label{eq:Product}
\end{equation}
The first term in \eqref{eq:Product} will be no worse than if we chose $k=n$. Therefore, the term in Line \eqref{eq:Product} is in fact bounded up to some constant by
\begin{equation}
    \left(\int_M\abs{\left(\lambda_{i}^{j}\right)^{-1}}^{n(p-1)}dV_{g_0}\right)^{\frac{1}{pn}}\left(\int_M\abs{\prod_{m\neq i}\lambda^j_{m}-1}^{\frac{k}{n-1}}dV_{g_0}\right)^{\frac{n-1}{pk}}\label{eq:Product2}
\end{equation}
The first term in \eqref{eq:Product2} is uniformly bounded in $j$ by Equation \eqref{eq:inverse_metric_convergence} and the second term tends to zero by \eqref{IntermediateConvergenceEq}.

We now analyze $\left\|\left(\lambda_{i}^{j}\right)^{-\frac{p-1}{p}}-1\right\|_{L^{p}}$ by first showing a relationship between inverse eigenvalues and the norm of the respective inverse metric by using H\"{o}lder's inequality for the counting measure and Proposition \ref{prop:reverse_triangle_ineq_snowflake}
\begin{align}
    \abs{ \frac{1}{\lambda_{i}^{j}}-1 } &\le  \sum_{i=1}^n\abs{ \frac{1}{\lambda_{i}^{j}} -1}
    \\&\le n^{\frac{4}{3}} \lp\sum_{i=1}^n\abs{ \frac{1}{\lambda_{i}^{j}}-1 } ^4 \rp^{\frac{1}{4}}
    \\&\le n^{\frac{4}{3}} \lp\sum_{i=1}^n\abs{\frac{1}{(\lambda_{i}^{j})^4}-1 }  \rp^{\frac{1}{4}} = n^{\frac{4}{3}}|g_j^{-1}-g_0^{-1}|_{g_0}^{\frac{1}{2}}. \label{eq:InverseNormToEigenvalues}
\end{align}
Now we can again use Proposition \ref{prop:reverse_triangle_ineq_snowflake} and \eqref{eq:InverseNormToEigenvalues} to find
\begin{align}
    \left\|\left(\lambda_{i}^{j}\right)^{-\frac{p-1}{p}}-1\right\|_{L^{p}}&= \left(\int_M \abs{\left(\lambda_{i}^{j}\right)^{-\frac{p-1}{p}}-1}^p dV_{g_0} \right)^{\frac{1}{p}}
    \\&\le \left(\int_M \abs{\left(\lambda_{i}^{j}\right)^{-1}-1}^{p-1} dV_{g_0} \right)^{\frac{1}{p}}
   \\ &\le n^{\frac{4(p-1)}{3p}} \left(\int_M |g_j^{-1}-g_0^{-1}|_{g_0}^{\frac{p-1}{2}} dV_{g_0} \right)^{\frac{1}{p}}.
\end{align}


Then by assumption \eqref{eq:inverse_metric_convergence} we see that $ \left\|P^{\frac{1}{p}}_{j}\left(\lambda_{i}^{j}\right)^{-1}-1\right\|_{L^{p}}\rightarrow 0$ and so we find
\begin{equation}
    \lp\lambda_{i}^{j}\rp^{\frac{1}{p}-1}\lp\prod_{m\neq i}\lambda^{j}_{m}\rp^{\frac{1}{p}}\stackrel{L^{p}}{\rightarrow}1,
\end{equation}
for $1 \le i \le n$. Hence by \eqref{eq:Final1}, \eqref{eq:Final2}, and \eqref{eq:Final3} we find
\begin{align}\label{WeakConvergenceConsequence}
    \||\nabla^{g_j} \phi|_{g_j}\|_{L^p(g_j)} \rightarrow  \||\nabla^{g_0} \phi|_{g_0}\|_{L^p(g_0)},
\end{align}

By Morrey's inequality and the density of $C^{\infty}(M,g_0)$ in $W^{1,p}(M,g_0)$ we can choose $\varphi \in C^{\infty}(M,g_0)$ so that
\begin{align}
    d_{p,g_0,g_0}^D(x_1,x_2)\le |\varphi(x_1)-\varphi(x_2)|+\varepsilon.
\end{align}
and
\begin{equation}
    \mathcal{H}_{p,g_0}(\phi)\leq D.
\end{equation}
Using equation \eqref{WeakConvergenceConsequence} we can pick $C_j$ so that $\varphi_j=C_j \varphi$ is a competitor for $g_j$ and $C_j \rightarrow 1$ as $j \rightarrow \infty$. Then, we see that
\begin{align}
    d_{p,g_j,g_0}^D(x_1,x_2) &\ge |\varphi_j(x_1)-\varphi_j(x_2)|
    \\&=C_j|\varphi(x_1)-\varphi(x_2)|=C_j(d_{p,g_0,g_0}^D(x_1,x_2)-\varepsilon).
\end{align}
Since this is true for all $\varepsilon>0$ we find 
\begin{align}\label{eq:FinalEq}
   C_j^{-1} d_{p,g_j,g_0}^D(x_1,x_2) &\ge d_{p,g_0,g_0}^D(x_1,x_2).
\end{align}
Now by taking the liminf on both sides of \eqref{eq:FinalEq} we find the desired result.
\end{proof}

 The second part of showing that pointwise convergence holds is to demonstrate an upper bound. This is more difficult, because now we must work with a sequence of test functions for the metrics $g_{i}$ instead of a single test function for the metric $g_0$. Analyzing such sequences requires using the Arzela-Ascoli theorem in some form.

\begin{thm}\label{thm:pointwiseBelow}
Let $M^n$ be a compact, connected, oriented manifold, $g$ a smooth Riemannian metric, and $g_j$ a sequence of continuous Riemannian metrics. If $n \ge 3$, $p > \frac{5n}{2}$, $\eta=\frac{5}{12}n$, 
\begin{align}
    \left( \int_M|g_j^{-1}|_{g_0}^{\frac{n\eta}{p-\eta}}dV_{g_0} \right)^{\frac{p-\eta}{2p}} &\le C,\label{eq:L^pMetricBound}
     \\\left( \int_M|g_j|_{g_0}^{\frac{n}{2}} dV_{g_0} \right)&\le C,\label{eq:L^nMetricBound}
\end{align}
and
\begin{align}\label{Eq:InverseAssumptionProp3.3}
\left(\int_M|g_j^{-1}-g_0^{-1}|_{g_0}^{\frac{p}{(p-1)}}dV_{g_0} \right)^{\frac{p-1}{p}} &\rightarrow 0, 
\end{align}
then for $x_1,x_2 \in M$ we find
\begin{align}
\limsup_{i \rightarrow \infty} d_{p,g_j,g_0}^D(x_1,x_2) \le  d_{p,g_0,g_0}^D(x_1,x_2).
\end{align}
\end{thm}
\begin{rmrk}
    We note that with the choice of $\eta=\frac{5}{12}n$ we find
    \begin{align}
        \frac{n \eta}{p-\eta}\le \frac{5}{7}n\le \frac{3}{2}n \le \frac{n}{2}n \le \frac{n(p-1)}{2},
    \end{align}
    and hence \eqref{LpMetricBounds} in Theorem \ref{thm:Main Theorem} implies the hypothesis of Theorem \ref{thm:pointwiseBelow}. We also note that for $n,p \ge 3$ we have
    \begin{align}
        \frac{p}{p-1}  \le 2 \le \frac{n(p-1)}{2},
    \end{align}
    which shows that \eqref{LpMetricBounds} of Theorem \ref{thm:Main Theorem} also implies \eqref{Eq:InverseAssumptionProp3.3} of Theorem \ref{thm:pointwiseBelow}.
\end{rmrk}

\begin{proof}
Let $\varphi_j \in W^{1,p}(M,g_j)$ so that
\begin{align}
d_{p,g_j,g_0}^D(x_1,x_2) - \varepsilon &\le |\varphi_j(x_1)-\varphi_j(x_2)|
\\ \int_M |\nabla^{g_j} \varphi_j |_{g_j}^p dV_{g_j}  &\le 1,\label{AdmissableIntegralConsequence}
\end{align}
and by definition
\begin{align}
|\varphi_j(x_1)-\varphi_j(x_2)| \le D d_{g_0}(x_1,x_2)^{\frac{p-n}{p}}.
\end{align}
Let $P_j=\sqrt{\det((g_j)_{lm})}$ where we note the estimate
\begin{align}
 P_j &\le n^{-\frac{n}{4}}|g_j|_{g_0}^{\frac{n}{2}},
\\ P_j^{-1} &\le n^{-\frac{n}{4}}|g_j^{-1}|_{g_0}^{\frac{n}{2}}.
\end{align}

Now for $1< \eta <p$ by applying Lemma \ref{lem:MetricCauchy} to $|\nabla ^{g_0} \varphi_j|^2_{g_0} = (g_0)^{-1}(d\varphi_j,d\varphi_j)$, the fact that by definition of the norms $|g_0^{-1}|_{g_j}=|g_j|_{g_0}$, and then H\"{o}lder's inequality we find
\begin{align}
    &\int_M|\nabla ^{g_0} \varphi_j|_{g_0}^{\eta} dV_{g_0} 
     \\&\qquad\le \int_M |\nabla ^{g_j}\varphi|_{g_j}^{\eta} |g_0^{-1}|_{g_j}^{\frac{\eta}{2}} dV_{g_0}
    \\&\qquad\le \int_M |\nabla ^{g_j}\varphi|_{g_j}^{\eta} |g_j|_{g_0}^{\frac{\eta}{2}} dV_{g_0}
    \\&\qquad= \int_M |\nabla ^{g_j}\varphi_j|_{g_j}^{\eta}P_j^{\frac{\eta}{p}} P_j^{-\frac{\eta}{p}} |g_j|_{g_0}^{\frac{\eta}{2}} dV_{g_0}
     \\&\qquad\le \left(\int_M |\nabla ^{g_j}\varphi_j|_{g_j}^{p}P_j dV_{g_0} \right)^{\frac{\eta} {p}} \left( \int_MP_j^{-\frac{\eta}{p-\eta}} |g_j|_{g_0}^{\frac{p\eta}{2(p-\eta)}} dV_{g_0} \right)^{\frac{p-\eta}{p}}
          \\&\qquad\le \left(\int_M |\nabla ^{g_j}\varphi_j|_{g_j}^{p}dV_{g_j} \right)^{\frac{\eta} {p}} \left( \int_M|g_j^{-1}|_{g_0}^{\frac{n\eta}{2(p-\eta)}} |g_j|_{g_0}^{\frac{p\eta}{2(p-\eta)}} dV_{g_0} \right)^{\frac{p-\eta}{p}}
           \\&\qquad\le  \left( \int_M|g_j^{-1}|_{g_0}^{\frac{n\eta}{p-\eta}}dV_{g_0} \right)^{\frac{p-\eta}{2p}} \left( \int_M|g_j|_{g_0}^{\frac{p\eta}{p-\eta}} dV_{g_0} \right)^{\frac{p-\eta}{2p}}\label{eq:Boundedness}
\end{align}
Notice that since $\eta=\frac{5n}{12}$ we see that
\begin{align}
     \frac{p\eta}{p-\eta} = \frac{p\frac{5n}{12}}{p-\frac{5n}{12}} = \frac{p}{p-\frac{5n}{12}}\frac{10}{12}\frac{n}{2}< \frac{n}{2},
\end{align}
implies that
\begin{align}
    \frac{p}{p-\frac{5n}{12}}<\frac{12}{10} \Rightarrow p> \frac{5n}{2}.
\end{align}
Since we have assumed $p$ is in this range we see that \eqref{eq:Boundedness} is bounded by \eqref{eq:L^pMetricBound} and \eqref{eq:L^nMetricBound}.

This implies that there exists a $\varphi_{\infty} \in W^{1,\eta}(M,g_0)\cap C^0(M,g_0)$ so that
\begin{align}
    \varphi_j &\rightharpoonup \varphi_{\infty}, \text{ in } W^{1,\eta}(M,g_0),
    \\ |\nabla ^{g_0}\varphi_j|_{g_0} &\rightharpoonup \Phi_{\infty}, \text{ in } L^{\eta}(M,g_0)\label{FurtherWeakConvergence},
\end{align}
where
\begin{align}
    |\nabla^{g_0}\varphi_{\infty}|_{g_0} \le \Phi_{\infty} \quad \text{ a.e.},
\end{align}
by Proposition \ref{prop: Technical Gradient Check}.

From \eqref{AdmissableIntegralConsequence} we also find there exists a $\psi_{\infty} \in L^p(M,g_0)$ so that
\begin{align}
    |\nabla^{g_j} \varphi_j|_{g_j} P_j^{\frac{1}{p}} \rightharpoonup  \psi_{\infty}  \text{ in } L^p,\quad  \int \psi_{\infty}^p dV_{g_0} \le 1. \label{AdmissableBound}
\end{align}
Now we notice if $\phi \in C^{\infty}(M,g_0)$ then
\begin{align}
&\left|\int_M(\Phi_{\infty}-\psi_{\infty}) P_j^{\frac{1}{p}} \phi dV_{g_0}\right|\label{eq:MainDeal}
\\&=\left|\int_M(\Phi_{\infty}-|\nabla \varphi_j|_{g_j}+|\nabla \varphi_j|_{g_j}-\psi_{\infty}) P_j^{\frac{1}{p}} \phi dV_{g_0}\right|
\\ &\le \left|\int_M \left(\Phi_{\infty}P_j^{\frac{1}{p}}-|\nabla \varphi_j|_{g_0}P_j^{\frac{1}{p}}+|\nabla \varphi_j|_{g_0}P_j^{\frac{1}{p}}-|\nabla \varphi_j|_{g_j}P_j^{\frac{1}{p}}\right) \phi dV_{g_0}\right|
\\&+ \left|\int_M\left(|\nabla \varphi_j|_{g_j}P_j^{\frac{1}{p}}-\psi_{\infty}+\psi_{\infty}-\psi_{\infty} P_j^{\frac{1}{p}} \right)\phi dV_{g_0}\right| 
\\ &\le\left|\int_M (\Phi_{\infty}-|\nabla \varphi_j|_{g_0})P_j^{\frac{1}{p}} \phi dV_{g_0}\right|
\\&+\left|\int_M\left(|\nabla \varphi_j|_{g_j}-|\nabla \varphi_j|_{g_0}\right)P_j^{\frac{1}{p}} \phi dV_{g_0}\right|
\\&+ \left|\int_M(|\nabla \varphi_j|_{g_j}P_j^{\frac{1}{p}}-\psi_{\infty})\phi dV_{g_0}\right| +\left| \int_M\psi_{\infty}(1- P_j^{\frac{1}{p}} )\phi dV_{g_0}\right|
\\&:=I+II+III+IV.
\end{align}
Notice that $III$ goes to zero by \eqref{AdmissableBound}. Our goal now is to show that $I$, $II$, and $IV$ are converging to zero as well. To this end we calculate
\begin{align}
  II&= \left|\int_M\left(|\nabla \varphi_j|_{g_j}-|\nabla \varphi_j|_{g_0}\right)P_j^{\frac{1}{p}}\phi dV_{g_0} \right| 
  \\&\le\int_M\left|\sqrt{g_j^{-1}(d\varphi_j,d\varphi_j)}-\sqrt{g_0^{-1}(d\varphi_j,d\varphi_j)}\right|P_j^{\frac{1}{p}} \phi dV_{g_0} 
  \\&\le\int_M\sqrt{\left|g_j^{-1}(d\varphi_j,d\varphi_j)-g_0^{-1}(d\varphi_j,d\varphi_j)\right|}P_j^{\frac{1}{p}} dV_{g_0} \label{Eq:AppliedSnowflake}
    \\&=\int_M\sqrt{\left|(g_j^{-1}-g_0^{-1})(d\varphi_j,d\varphi_j)\right|}P_j^{\frac{1}{p}} \phi dV_{g_0}
\end{align}
where we used Proposition \ref{prop:reverse_triangle_ineq_snowflake} in \eqref{Eq:AppliedSnowflake}. By Further estimating we find
\begin{align}
  II&\le\int_M|g_j^{-1}-g_0^{-1}|_{g_j}^{\frac{1}{2}}|\nabla \varphi_j|_{g_j} P_j^{\frac{1}{p}} \phi dV_{g_0}\label{Eq:AppliedMetricCauchy}
      \\&\le\int_M|g_j^{-1}-g_0^{-1}|_{g_0}^{\frac{1}{2}}|g_j|_{g_0}^{\frac{1}{4}}|\nabla \varphi_j|_{g_j} P_j^{\frac{1}{p}} \phi dV_{g_0}\label{Eq:AppliedMetricCauchyAgain}
    \\&\le \left(\int_M|g_j^{-1}-g_0^{-1}|_{g_0}^{\frac{p}{2(p-1)}}|g_j|_{g_0}^{\frac{p}{4(p-1)}}dV_{g_0} \right)^{\frac{p-1}{p}}\left(\int_M|\nabla \varphi_j|_{g_j} ^pdV_{g_j}\right)^{\frac{1}{p}},\label{Eq:AppliedHolders}
    \\&\le \left(\int_M|g_j^{-1}-g_0^{-1}|_{g_0}^{\frac{p}{(p-1)}}dV_{g_0} \right)^{\frac{p-1}{2p}}\label{Eq:AppliedHolders2}
    \\& \quad \cdot\left(\int_M|g_j|_{g_0}^{\frac{p}{2(p-1)}}dV_{g_0} \right)^{\frac{p-1}{2p}}\left(\int_M|\nabla \varphi_j|_{g_j} ^pdV_{g_j}\right)^{\frac{1}{p}},
\end{align}
where we used Lemma \ref{lem:MetricCauchy} in \eqref{Eq:AppliedMetricCauchy}, and H\"{o}lder's inequality in \eqref{Eq:AppliedHolders}, \eqref{Eq:AppliedHolders2}. Note that the first term on the right hand side of \eqref{Eq:AppliedHolders2} tends to zero by \eqref{Eq:InverseAssumptionProp3.3}, the second term is uniformly bounded by assumption since for $p, n \ge 3$ we see that $\frac{p}{2(p-1)}\le \frac{3}{4} \le \frac{n}{2}$, and the third term is uniformly bounded since it weakly converges in $L^p$ by \eqref{AdmissableBound}.

\begin{align}
  IV&=  \left| \int_M\psi_{\infty}(1- P_j^{\frac{1}{p}} )\phi dV_{g_0}\right|
  \\&\le   \int_M\psi_{\infty}\left|1- P_j^{\frac{1}{p}} \right|\phi dV_{g_0}
    \\&\le  \lp\max \phi\rp \int_M\psi_{\infty}\left|1- P_j \right|^{\frac{1}{p}} dV_{g_0}
    \\&\le  \lp\max \phi\rp \lp\int_M\psi_{\infty}^p dV_{g_0} \rp^{\frac{1}{p}}\lp\int_M\left|1- P_j \right|^{\frac{1}{p-1}} dV_{g_0}\rp^{\frac{p-1}{p}}.\label{eq:BoundThis}
\end{align}

Now we observe that we can rewrite 
\begin{align}
    \lambda_1^i...\lambda_n^i-1&=\lambda_1^i...\lambda_n^i-\lambda_2^i...\lambda_n^i+\lambda_2^i...\lambda_n^i-...-\lambda_n^i+\lambda_n^i-1
    \\&=(\lambda_1^i-1)\lambda_2^i...\lambda_n^i+...+(\lambda_{n-1}^i-1)\lambda_n^i+(\lambda_n^i-1)
\end{align}

If we pick $n>k\ge n-1$, then using the triangle inequality for $|\cdot|^{\frac{1}{p-1}}$ and the generalized H\"{o}lder inequality we find 
\begin{align}
    &\left\| \bigl|\lambda_1^j...\lambda_n^j-1\bigr|^{\frac{1}{p-1}}\right\|_{L^1}
    \\ \le &\left\|\bigl|(\lambda_1^j-1)\lambda_2^j...\lambda_n^j\bigr|^{\frac{1}{p-1}}\right\|_{L^1}+...+\left\|\bigl|\lambda_n^j-1\bigr|^{\frac{1}{p-1}}\right\|_{L^1}
    \\\le &\left\|\bigl|\lambda_1^j-1\bigr|^{\frac{1}{p-1}}\right\|_{L^n}\left\|\bigl|\lambda_2^j\bigr|^{\frac{1}{p-1}}\right\|_{L^n}...\left\|\bigl|\lambda_n^j\bigr|^{\frac{1}{p-1}}\right\|_{L^n}
    \\&\quad +\left\|\bigl|\lambda_2^j-1\bigr|^{\frac{1}{p-1}}\right\|_{L^{n-1}}\left\|\bigl|\lambda_3^j\bigr|^{\frac{1}{p-1}}\right\|_{L^{n-1}}...\left\|\bigl|\lambda_n^j\bigr|^{\frac{1}{p-1}}\right\|_{L^{n-1}}
    \\&\quad +...+\left\|\bigl|\lambda_n^j-1\bigr|^{\frac{1}{p-1}}\right\|_{L^1}.
\end{align}

Now we want to relate convergence of $\lambda_{i}^{j}$ to convergence of $(\lambda_{i}^{j})^{-1}$. To this end, we have
\begin{equation}
    \left|1-\lambda_{i}^{j}\right|=\left(\lambda_{i}^{j}\right)\left|\left(\lambda_{i}^{j}\right)^{-1}-1\right|.
\end{equation}
We now apply H{\"o}lder's inequality with $p-1$ and $\frac{p-1}{p-2}$ to find
\begin{equation}
    \int_M\abs{1-\lambda_{i}^{j}}^{\frac{n}{p-1}}dV_{g_0}\leq\sup_{ij}\|\lambda_{i}^{j}\|^{\frac{n}{p-1}}_{L^n}\lp\int_M\abs{(\lambda_{i}^{j})^{-1}-1}^{\frac{n}{p-2}}dV_{g_0}\rp^{\frac{p-2}{p-1}}\label{eq:LpTrick}
\end{equation}
Therefore, if $\sup_{ij}\|\lambda_{i}^{j}\|_{L^{\frac{n}{p-1}}}<\infty$ and $(\lambda_{i}^{j})^{-1}$ converges to $1$ in $L^{\frac{n}{p-2}}$, then \eqref{eq:LpTrick} shows that $\lambda_{i}^{j}$ converges to $1$ in $L^{\frac{n}{p-1}}$. Now we notice that if $p >\frac{5n}{2}> n+2$ for $n \ge 3$ then
\begin{align}
    \frac{n}{p-2}<\frac{p-2}{p-2}=1< \frac{p}{p-1},
\end{align}
and hence \eqref{Eq:InverseAssumptionProp3.3} implies the required convergence and \eqref{eq:L^nMetricBound} implies the required bound to imply that \eqref{eq:BoundThis} is getting small.

Now if we choose $\eta = \frac{5n}{12} > 1$ for $n \ge 3$ and $p>\frac{5n}{2} > \frac{25n}{12}$ then by \eqref{FurtherWeakConvergence}  we find
\begin{align}
    I&=\left|\int_M (\Phi_{\infty}-|\nabla \varphi_j|_{g_0})P_j^{\frac{1}{p}} \phi dV_{g_0}\right| 
    \\& \le \left|\int_M (\Phi_{\infty}-|\nabla \varphi_j|_{g_0})(P_j^{\frac{1}{p}}-1) \phi dV_{g_0}\right|+ \left|\int_M (\Phi_{\infty}-|\nabla \varphi_j|_{g_0}) \phi dV_{g_0}\right|
    \\& \le \left|\lp\int_M (\Phi_{\infty}-|\nabla \varphi_j|_{g_0})^{\eta}dV_{g_0}\rp^{\frac{1}{\eta}} \lp\int_M(P_j^{\frac{1}{p}}-1)^{\eta'} \phi dV_{g_0}\rp^{\frac{1}{\eta'}}\right|
    \\&+ \left|\int_M (\Phi_{\infty}-|\nabla \varphi_j|_{g_0}) \phi dV_{g_0}\right| 
     \\& \le \left|\lp\int_M (\Phi_{\infty}-|\nabla \varphi_j|_{g_0})^{\eta}dV_{g_0}\rp^{\frac{1}{\eta}} \lp\int_M(P_j-1)^{\frac{\eta'}{p}} \phi dV_{g_0}\rp^{\frac{1}{\eta'}}\right|\label{eq:LastToZero1}
    \\&+ \left|\int_M (\Phi_{\infty}-|\nabla \varphi_j|_{g_0}) \phi dV_{g_0}\right|.\label{eq:LastToZero2}
\end{align}
Here we see that \eqref{eq:LastToZero2} tends to zero by \eqref{FurtherWeakConvergence}. Since $\eta=\frac{5n}{12}$ we see that $\eta'=\frac{5n-12}{5n}<1$ and $\frac{\eta'}{p}<\frac{1}{p}<\frac{n}{p-2}$ and so by \eqref{eq:LpTrick}, \eqref{FurtherWeakConvergence} and \eqref{Eq:InverseAssumptionProp3.3} we see that \eqref{eq:LastToZero1} tends to zero. 

Since in \eqref{eq:MainDeal} we have vanishing for every $\phi$, we find $\Phi_{\infty}=\psi_{\infty}$ a.e. So by combining with \eqref{AdmissableBound} we find
\begin{align}
\int_M |\nabla^{g_0} \varphi_{\infty}|_{g_0}^p  dV_{g_0} \le \int_M\Phi_{\infty}^p dV_{g_0} \le 1,
\end{align}
which implies that $\varphi_{\infty}$ is an admissible function for $d_{p,g_0,g_0}^D$.

We also notice that since $\varphi_j \rightarrow \varphi_{\infty}$ in $C^0$ we have
\begin{align}
    |\varphi_j(x_1)-\varphi_j(x_2)|\rightarrow  |\varphi_{\infty}(x_1)-\varphi_{\infty}(x_2)|,
\end{align}
which implies
\begin{align}
    \limsup_{j \rightarrow \infty} d_{p,g_j,g}^D(x_1,x_2)- \varepsilon &\le  \limsup_{j \rightarrow \infty} |\varphi_j(x_1)-\varphi_j(x_2)| 
    \\&= |\varphi_{\infty}(x_1)-\varphi_{\infty}(x_2)| \le d_{p,g_0,g_0}^D(x_1,x_2).
\end{align}
\end{proof}

\begin{proof}[Proof of Theorem \ref{thm:Main Theorem}]
    From Lemma \ref{lem:compactness_of_distance_functions}, a subsequence of $d_{p,g_i,g_0}^D$ converges uniformly to a symmetric function $d:M\times M \rightarrow [0,\infty)$ which satisfies the triangle inequality. By combining Theorem \ref{LimInfEst} with Theorem \ref{thm:pointwiseBelow} we see pointwise convergence of $d_{p,g_i,g_0}^D$ to $d_{p,g_0}=d_{p,g_0,g_0}^D$ for $D$ chosen to be larger than $\bar{D}$, the Morrey constant for $g_0$. Hence $d=d_{p,g_0,g_0}^D$ and the original subsequence must not have been necessary since every convergence subsequence has the same limit. So we find that $d_{p,g_i,g_0}^D$ converges uniformly to $d_{p,g_0,g_0}^D$, as desired.
\end{proof}

    \bibliographystyle{plain}
    \bibliography{bibliography}
    
\end{document}